[1994/12/01]
\documentclass{ijmart-mod}
\chardef\bslash=`\\ 

\usepackage{graphicx}
\usepackage[breaklinks=true]{hyperref}
\usepackage{mathtools}
\usepackage{xcolor}
\usepackage{mathabx}
\usepackage{array}
\usepackage{multirow}

\newtheorem{thm}{Theorem}[section]

\newtheorem{lem}[thm]{Lemma}
\newtheorem{prop}[thm]{Proposition}

\theoremstyle{definition}

\newtheorem{rem}[thm]{Remark}

\theoremstyle{remark}

\newcommand{\eval}[2][\right]{\relax
  \ifx#1\right\relax \left.\fi#2#1\rvert}

\begin{document}
\title{Kissing polytopes}

\author[A. Deza]{Antoine Deza}
\address{McMaster University, Hamilton, Ontario, Canada}
\email{deza@mcmaster.ca} 

\author[S. Onn]{Shmuel Onn}
\address{Technion -- Israel Institute of Technology, Haifa, Israel}
\email{onn@technion.ac.il}

\author[S. Pokutta]{Sebastian Pokutta}
\address{Zuse Institute Berlin, Germany}
\email{pokutta@zib.de}

\author[L. Pournin]{Lionel Pournin}
\address{Universit{\'e} Paris 13, Villetaneuse, France}
\email{lionel.pournin@univ-paris13.fr}

\begin{abstract}
We investigate the following question: \emph{how close can two disjoint lattice polytopes contained in a fixed hypercube be?} This question stems from various contexts where the minimal distance between such polytopes appears in complexity bounds of optimization algorithms. We provide nearly matching lower and upper bounds on this distance and discuss its exact computation. We also give similar bounds in the case of disjoint rational polytopes whose binary encoding length is prescribed.
\end{abstract}
\maketitle


\section{Introduction}\label{DHLOPP.sec.0}

In general, the distance between two disjoint convex bodies $P$ and $Q$ contained in $\mathbb{R}^d$ can get arbitrarily small. However, this is no longer the case when $P$ and $Q$ satisfy certain constraints. For instance, if $P$ and $Q$ are two $d$-dimensional $0/1$\nobreakdash-polytopes, then they cannot be closer than a positive distance that only depends on $d$. This is due to the observation that, when $d$ is fixed, there are finitely many such pairs of polytopes. Another relevant constraint that often arises in optimization algorithms is when $P$ and $Q$ are rational polytopes whose binary encoding length (as subsets of $\mathbb{R}^d$ satisfying a set of linear inequalities) is prescribed. Here, again, the smallest possible distance between $P$ and $Q$ is a positive number that depends on that encoding length and on $d$. Our goal is to estimate these minimal distances.

Our study stems from the complexity bounds established by G{\'a}bor Braun, Sebastian Pokutta, and Robert Weismantel~\cite{BraunPokuttaWeismantel2022}. They provide an algorithm that computes a point in $P\cap{Q}$ when that intersection is non-empty or certifies that $P\cap{Q}$ is empty. In the latter case, the number of calls to a linear optimization oracle over $P$ and $Q$ required to certify that $P\cap{Q}$ is empty is
$$
O\Biggl(\frac{1}{d(P,Q)^2}\Biggr)
$$
and therefore, it is natural to ask how small $d(P,Q)$ can get.  

Our study is related to the notion of \emph{facial distance} considered by Javier Pe{\~n}a and Daniel Rodriguez~\cite[Section~2]{PenaRodriguez2018} and David Gutman and Javier Pe{\~n}a~\cite{GutmanPena2018, Pena2019}. It is also linked to the \emph{vertex-facet distance} of a polytope investigated by Amir Beck and Shimrit Shtern \cite{BeckShtern2017} and to the closely related \emph{pyramidal width} studied by Simon Lacoste-Julien and Martin Jaggi~\cite{Lacoste-JulienJaggi2015} and by Luis Rademacher and Chang Shu~\cite{RademacherShu2022}. We refer the reader to the survey by G{\'a}bor Braun, Alejandro Carderera, Cyrille Combettes, Hamed Hassani, Amin Karbasi, Aryan Mokhtari, and Sebastian Pokutta~\cite{BraunCardereraCombettesHassaniKarbasiMokhtariPokutta2022} for an overview of these notions.

The facial distance is crucial in establishing linear convergence rates for conditional gradient methods over polytopes and naturally occurs in the complexity bounds. The facial distance of a polytope $P$ is defined as 
\begin{equation}\label{eq:fd}
\Phi(P)=\min\Bigl\{d\bigl(F,\mathrm{conv}(\mathcal{V} \setminus F)\bigr):F\in\mathcal{F}\Bigr\},
\end{equation}
where $\mathcal{V}$ denotes the vertex set of $P$ and $\mathcal{F}$ the set of its proper faces. In other words, the facial distance of $P$ is the minimal distance between any of its faces and the convex hull of its vertices not contained in that face. In contrast to our study, this notion considers a specific polytope $P$ and decomposes it into its faces and their complements. 
The vertex-facet distance is measured in the special case when $F$ is a facet of the considered polytope and is replaced in \eqref{eq:fd} with its affine hull. It can then be expressed as
\begin{equation}\label{eq:vfd}
\Delta(P)=\min\Bigl\{d\bigl(\mathrm{aff}(F),\mathrm{conv}(\mathcal{V} \setminus F)\bigr):F\in\overline{\mathcal{F}}\Bigr\},
\end{equation}
where $\overline{\mathcal{F}}$ is the set of the facets of $P$, as shown in \cite[Section 2]{PenaRodriguez2018}. Bounds have been given on the smallest possible vertex-facet distance of $0/1$-simplices \cite{AlonVu1997,GrahamSloane1984}. In particular, Noga Alon and V\u{a}n V{\~u} show~\cite[Theorem 3.2.2]{AlonVu1997} that
\begin{equation}\label{DHLOPP.sec.0.eq.1}
\frac{1}{\sqrt{2}^{d\,\mathrm{log}\,d-2d+o(d)}}\leq\min\,\Delta(S)\leq\frac{1}{\sqrt{2}^{d\,\mathrm{log}\,d-4d+o(d)}}
\end{equation}
where the minimum is over all the $d$-dimensional $0/1$-simplices $S$.


The results of Stephen Vavasis on the complexity of quadratic optimization~\cite{Vavasis1990}, generalized by Alberto Del Pia, Santanu Dey, and Marco Molinaro in~\cite{DelPiaDeyMolinaro2017} imply as a special case that the squared distance between two rational polytopes is a rational number. Our work is concerned by providing bounds on how close such polytopes can be under the mentioned constraints. 

Recall that a polytope whose vertices belong to the integer lattice $\mathbb{Z}^d$ is a \emph{lattice polytope}. We will refer to a lattice polytope contained in the hypercube $[0,k]^d$ as a \emph{lattice $(d,k)$-polytope}. In this article, we first provide a lower bound, as a function of $d$ and $k$ on the smallest possible distance between two disjoint lattice $(d,k)$-polytopes and then we complement these lower bounds with constructions that provide almost matching upper bounds. 

In terms of lower bounds our main result is the following.

\begin{thm}\label{DHLOPP.sec.0.thm.1}
If $P$ and $Q$ are disjoint lattice $(d,k)$-polytopes, then 
$$
d(P,Q)\geq\frac{1}{(kd)^{2d}}\mbox{.}
$$
\end{thm}

We shall in fact prove a stronger bound (see Theorem \ref{DHLOPP.sec.1.thm.1}) which Theorem~\ref{DHLOPP.sec.0.thm.1} is a consequence of. We also prove a lower bound on the distance of two rational polytopes in terms of the dimension and their binary encoding length (see Theorem~\ref{DHLOPP.sec.1.thm.4}). Our main result regarding upper bounds in the following. 

\begin{thm}\label{DHLOPP.sec.0.thm.2}
Consider a positive integer $k$. For any large enough $d$, there exist two disjoint $(d,k)$-lattice polytopes $P$ and $Q$ such that
$$
d(P,Q)\leq\frac{1}{\bigl(k\sqrt{d}\bigr)^{\!\!\sqrt{d}}}\mbox{.}
$$
\end{thm}

As above, Theorem \ref{DHLOPP.sec.0.thm.2} follows from a stronger bound (see Theorem \ref{DHLOPP.sec.2.thm.1}). We also give an upper bound on the smallest possible distance between two rational polytopes whose binary encoding length is prescribed (see Theorem \ref{DHLOPP.sec.2.thm.3}). 

By its definition, the facial distance of a polytope is a distance between two polytopes. Inversely, the distance between two polytopes $P$ and $Q$ is the distance between two of their faces that belong to parallel hyperplanes. In particular, $d(P,Q)$ is at least the facial distance of the convex hull of these two faces. As a consequence, our results provide bounds on the smallest possible facial distance of a lattice $(d,k)$-polytope in terms of $d$ and $k$.

\begin{thm}\label{DHLOPP.sec.0.thm.3}
For any positive $k$ and large enough $d$,
\begin{equation}\label{DHLOPP.sec.0.eq.2}
\frac{1}{(kd)^{2d}}\leq\mathrm{min}\,\Phi(P)\leq\frac{1}{\bigl(k\sqrt{d}\bigr)^{\!\!\sqrt{d}}}
\end{equation}
where the minimum is over all the lattice $(d,k)$-polytopes $P$.
\end{thm}

Similar bounds in the case of rational polytopes, in terms of their dimension and binary encoding length follow from Theorems \ref{DHLOPP.sec.1.thm.4} and \ref{DHLOPP.sec.2.thm.3}. 


We establish the announced lower bounds for lattice polytopes in Section~\ref{DHLOPP.sec.1}. The upper bounds and the corresponding constructions for lattice polytopes are provided in Section~\ref{DHLOPP.sec.2}. These upper bounds are only valid for all sufficiently large dimensions and we provide bounds in Section~\ref{DHLOPP.sec.3} that hold in all dimensions. In the same section, we study the smallest possible distance of two lattice polytopes whose dimension is fixed independently on the dimension of the ambient space. Section~\ref{DHLOPP.sec.4} contains computational results. We report in that section the exact value of the smallest possible distance between disjoint lattice $(d,k)$\nobreakdash-polytopes for certain $d$ and $k$ (see Table \ref{DHLOPP.sec.4.table.1}). In order to compute these distances, we prove in Section~\ref{DHLOPP.sec.4} that one can restrict to considering a well-behaved subset of the pairs of lattice $(d,k)$-polytopes. We end the article with Section~\ref{DHLOPP.sec.1.5} where our lower and upper bounds are given on the smallest possible distance of two rational polytopes in terms of their binary encoding length.

\section{Lower bounds}\label{DHLOPP.sec.1}

In this section $P$ and $Q$ are two fixed, disjoint polytopes contained in $\mathbb{R}^d$ and our goal is to prove Theorem~\ref{DHLOPP.sec.0.thm.1}. Let us first introduce some notations and give a few remarks. Since $P$ and $Q$ are compact subsets of $\mathbb{R}^d$, there exists a point $p$ in $P$ and a point $q$ in $Q$ whose distance is equal to $d(P,Q)$. Let $f_P$ denote the unique face of $P$ that contains $p$ in its relative interior and $f_Q$ the unique face of $Q$ that contains $q$ in its relative interior. We remark that $f_P$ and $f_Q$ are contained in two parallel hyperplanes orthogonal to $p-q$. This situation is illustrated in Figure~\ref{DHLOPP.sec.1.fig.1}, where $P$ and $Q$ are two $0/1$-polytopes, $f_P$ is the diagonal of the cube, and $f_Q$ a diagonal of one of its square faces.

We now consider $\mathrm{dim}(f_P)+1$ vertices of $f_P$, that we label by $u^0$ to $u^{\mathrm{dim}(f_P)}$ such that the vectors $u^1-u^0$ to $u^{\mathrm{dim}(f_P)}-u^0$ are linearly independent. Similarly, pick a family $v^0$ to $v^{\mathrm{dim}(f_Q)}$ of vertices of $f_Q$ such that the vectors $v^1-v^0$ to $v^{\mathrm{dim}(f_Q)}-v^0$ are linearly independent. Consider the set 
$$
S=\Bigl\{u^i-u^0:1\leq{i}\leq\mathrm{dim}(f_P)\Bigr\}\cup\Bigl\{v^i-v^0:1\leq{i}\leq\mathrm{dim}(f_Q)\Bigr\}
$$
and extract from it a subset of linearly independent vectors $w^1$ to $w^r$ that span the same subspace of $\mathbb{R}^d$ than $S$. Observe that $r$ is at most $d-1$ because these vectors are linearly independent and all of them are orthogonal to $p-q$. Further denote by $w^0$ the difference $u^0-v^0$. 
\begin{figure}[t]
\begin{centering}
\includegraphics[scale=1]{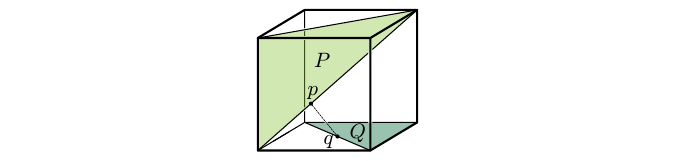}
\caption{Two $0/1$-polytopes $P$ and $Q$ and points $p$ and $q$ such that $d(P,Q)$ is equal to $d(p,q)$.}\label{DHLOPP.sec.1.fig.1}
\end{centering}
\end{figure}
Since $f_P$ and $f_Q$ are contained into two parallel hyperplanes orthogonal to $p-q$, the scalar product $(x-y)\mathord{\cdot}(p-q)$, where $x$ belongs to $f_P$ and $y$ to $f_Q$ does not depend on which points $x$ and $y$ are chosen in $f_P$ and in $f_Q$. As a consequence, the equality
$$
d(p,q)=\frac{(p-q)\mathord{\cdot} (p-q)}{\|p-q\|}
$$
can be rewritten into
\begin{equation}\label{DPP.sec.1.equation.1}
d(p,q)=w^0\mathord{\cdot}\frac{(p-q)}{\|p-q\|}\mbox{.}
\end{equation}

We will express the quotient in the right-hand side of (\ref{DPP.sec.1.equation.1}) using the vectors $w^i$. In order to do that, consider the $r\mathord{\times}r$ matrix $M$ whose term in row $i$ and column $j$ is $w^i\mathord{\cdot}w^j$, the column vector $b$ whose coordinates are $w^0\mathord{\cdot}w^1$ to $w^0\mathord{\cdot}w^r$, and the matrix $M_i$ obtained from $M$ by replacing column $i$ with $b$.

\begin{lem}\label{DPP.sec.0.lem.1}
The distance between $P$ and $Q$ satisfies
$$
d(P,Q)=w^0\mathord{\cdot}\frac{a}{\|a\|}
$$
where
$$
a=\mathrm{det}(M)w^0-\sum_{i=1}^r\mathrm{det}(M_i)w^i\mbox{.}
$$
\end{lem}
\begin{proof}
Observe that $p-q$ belongs to the space spanned by vectors $w^0$ to $w^r$. Hence, there exist $r+1$ coefficients $\alpha_0$ to $\alpha_r$ such that
\begin{equation}\label{DPP.sec.0.lem.1.eq.1}
p-q=\sum_{i=0}^r\alpha_iw^i\mbox{.}
 \end{equation}
 
Let $j$ be an integer such that $1\leq{j}\leq{r}$. As $w^j$ is orthogonal to $p-q$,
\begin{equation}\label{DPP.sec.0.lem.1.eq.2}
\sum_{i=0}^r\alpha_i(w^i\mathord{\cdot}w^j)=0\mbox{.}
\end{equation}

Since $p-q$ is non-zero and orthogonal to the vectors $w^1$ to $w^r$, it cannot be a linear combination of these vectors. It immediately follows that $\alpha_0$ is non-zero and we can denote, for each integer $i$ such that $1\leq{i}\leq{r}$,
\begin{equation}\label{DPP.sec.0.lem.1.eq.2.5}
\beta_i=-\frac{\alpha_i}{\alpha_0}\mbox{.}
\end{equation}

With this notation, (\ref{DPP.sec.0.lem.1.eq.2}) can be rewritten into
$$
\sum_{i=1}^r\beta_i(w^i\mathord{\cdot}w^j)=w^0\mathord{\cdot}w^j
$$
and the linear system obtained by letting $j$ range between $1$ and $r$ is
$$
M\beta=b\mbox{,}
$$
where $\beta$ is the column vector whose coordinates are $\beta_1$ to $\beta_r$. Observe that $M$ has rank $r$ since the vectors $w^i$ are linearly independent and that its determinant is therefore non-zero. As a consequence, according to Cramer's rule,
$$
\beta_i=\frac{\mathrm{det}(M_i)}{\mathrm{det}(M)}
$$
and, by (\ref{DPP.sec.0.lem.1.eq.2.5}), one can rewrite (\ref{DPP.sec.0.lem.1.eq.1}) into
\begin{equation}\label{DPP.sec.0.lem.1.eq.3}
\lambda(p-q)=\mathrm{det}(M)w^0-\sum_{i=0}^r\mathrm{det}(M_i)w^i\mbox{,}
\end{equation}
where
$$
\lambda=\frac{\mathrm{det}(M)}{\alpha_0}\mbox{.}
$$

Finally observe that (\ref{DPP.sec.1.equation.1}) can be rewritten into
$$
d(p,q)=w^0\mathord{\cdot}\frac{\lambda(p-q)}{\|\lambda(p-q)\|}\mbox{.}
$$

Combining this with (\ref{DPP.sec.0.lem.1.eq.3}) proves the lemma.
\end{proof}

Now observe that, when $P$ and $Q$ are rational polytopes, then the vectors $w^0$ to $w^r$ have rational coordinates. Therefore, we recover the following remark from Lemma~\ref{DPP.sec.0.lem.1}. This remark is a consequence of a more general result due to Stephen Vavasis~\cite{Vavasis1990} that was further improved in \cite{DelPiaDeyMolinaro2017}.

\begin{rem}\label{DPP.sec.0.rem.1}
If $P$ and $Q$ are rational polytopes then $d(P,Q)^2$ is rational.
\end{rem}

We are now ready to prove the announced bound on on $d(P,Q)$ in the case when both $P$ and $Q$ are lattice polytopes.

\begin{thm}\label{DHLOPP.sec.1.thm.1}
If $P$ and $Q$ are disjoint lattice $(d,k)$-polytopes, then 
$$
d(P,Q)\geq\frac{1}{k^{2d-1}\sqrt{d}^{3d}}\mbox{.}
$$
\end{thm}
\begin{proof}
According to Lemma \ref{DPP.sec.0.lem.1},
\begin{equation}\label{DHLOPP.sec.1.thm.1.eq.1}
d(P,Q)=\frac{w^0\mathord{\cdot}a}{\|a\|}
\end{equation}
where
\begin{equation}\label{DHLOPP.sec.1.thm.1.eq.2}
a=\mathrm{det}(M)w^0-\sum_{i=1}^r\mathrm{det}(M_i)w^i\mbox{.}
\end{equation}

Assuming that $P$ and $Q$ are lattice $(d,k)$-polytopes, the vectors $w^0$ to $w^r$ have integers coordinates. It then follows from (\ref{DHLOPP.sec.1.thm.1.eq.2}) that all the coordinates of $a$ are integers. By the assumption that $P$ and $Q$ are disjoint, the numerator in the right-hand side of (\ref{DHLOPP.sec.1.thm.1.eq.1}) is then at least $1$. As a consequence,
\begin{equation}\label{DHLOPP.sec.1.thm.1.eq.3}
d(P,Q)\geq\frac{1}{\|a\|}\mbox{.}
\end{equation}

Since  $P$ and $Q$ are lattice $(d,k)$\nobreakdash-polytopes, all the $w^i$ are contained in the hypercube $[-k,k]^d$. Hence the absolute value of each entry in the matrices $M$ and $M_i$ is at most $dk^2$ and, by Hadamard's inequality,
$$
|\mathrm{det}(M_i)|\leq{d^rk^{2r}}r^\frac{r}{2}
$$
for all $i$. Moreover, the same inequality holds when $M_i$ is replaced by $M$ in the left-hand side. Plugging this into (\ref{DHLOPP.sec.1.thm.1.eq.2}) yields
$$
|a_i|\leq{(r+1)d^rk^{2r+1}}r^\frac{r}{2}.
$$

It follows that
$$
\|a\|\leq{(r+1)d^{\frac{2r+1}{2}}k^{2r+1}}r^\frac{r}{2}
$$
and according to (\ref{DHLOPP.sec.1.thm.1.eq.3}),
$$
d(P,Q)\geq\frac{1}{(r+1)d^{\frac{2r+1}{2}}k^{2r+1}r^\frac{r}{2}}\mbox{.}
$$

Finally, recall that $r$ is at most $d-1$. Hence, this implies
$$
d(P,Q)\geq\frac{1}{k^{2d-1}d^\frac{3d}{2}}
$$
as desired.
\end{proof}

Note that the distance between the origin of $\mathbb{R}^d$ and the $(d-1)$-dimensional standard simplex is equal to $1/\sqrt{d}$. It turns out that the distance between the origin and any lattice polytope contained in the positive orthant $[0,+\infty[^d$ but that does not contain the origin is at least this value.

\begin{lem}\label{origin}
If $P$ is a lattice polytope contained in $[0,+\infty[^d\mathord{\setminus}\{0\}$, then
$$
d(0,P)\geq\frac{1}{\sqrt{d}}\mbox{.}
$$ 
\end{lem}
\begin{proof}
Let $p$ be a point in $P$ such that $d(0,P)=d(0,p)$. Observe that all the vertices $x$ of $P$ satisfy $\|x\|_1\geq1$. As any point in $P$ is a convex combination of vertices of $P$, it follows that $\|p\|_1\geq1$. However, by the Cauchy--Schwarz inequality $\|p\|_1\leq\sqrt{d}\|p\|_2$, which proves the lemma.
\end{proof}

\section{Upper bounds}\label{DHLOPP.sec.2}

In this section, $k$ is fixed and we further consider two positive integers $\sigma$ and $\delta$. Most of this section is devoted to prove the following theorem.

\begin{thm}\label{DHLOPP.sec.2.thm.2}
If $\delta$ is at least $4$, then there exist two lattice $(d,k)$\nobreakdash-polytopes $P$ and $Q$, where $d$ is equal to $\delta(\sigma+1)$, such that
\begin{equation}\label{DHLOPP.sec.2.eq.1}
d(P,Q)\leq\frac{\sqrt{\delta\sigma}}{\bigl(\!k(\delta-1)\bigr)^{\!\sigma}}\mbox{.}
\end{equation}
\end{thm}

Before we build the two polytopes $P$ and $Q$ that appear in the statement of Theorem~\ref{DHLOPP.sec.2.thm.2}, let us state and prove the main result of the section, which we obtain as a consequence of this theorem.

\begin{thm}\label{DHLOPP.sec.2.thm.1}
Consider a number $\alpha$ in $]0,1[$. For any large enough $d$, there exist two disjoint lattice $(d,k)$-polytopes $P$ and $Q$ such that
$$
d(P,Q)\leq\frac{1}{k^{d^\alpha}d^{(1-\alpha)d^\alpha}}\mbox{.}
$$
\end{thm}
\begin{proof}
Let $\beta$ be a number in the interval $]\alpha,1[$. Assume that
\begin{equation}\label{DHLOPP.sec.2.thm.1.eq.1}
d\geq8^{\frac{1}{1-\beta}}
\end{equation}
and denote
\begin{equation}\label{DHLOPP.sec.2.thm.1.eq.2}
\left\{
\begin{array}{l}
\sigma=\left\lfloor{d^{\beta}}\right\rfloor\mbox{ and}\\[\medskipamount]
\delta=\displaystyle\left\lfloor\frac{d}{\sigma+1}\right\rfloor\!\!\mbox{.}
\end{array}
\right.
\end{equation}

Observe that $\sigma$ is at least $1$. In addition, (\ref{DHLOPP.sec.2.thm.1.eq.1}) can be rewritten into
$$ 
d\geq{8d^\beta}\mbox{.}
$$

As $d^\beta$ is at least $1$, it follows that
$$
d\geq4d^{\beta}+4
$$
and as a consequence, $\delta$ is at least $4$.

According to Theorem \ref{DHLOPP.sec.2.thm.2}, under these conditions on $\sigma$ and $\delta$, there exist two lattice $(\delta(\sigma+1),k)$\nobreak-polytopes $P$ and $Q$ such that
\begin{equation}\label{DHLOPP.sec.2.thm.1.eq.3}
d(P,Q)\leq\frac{\sqrt{\delta\sigma}}{\bigl(k(\delta-1)\bigr)^{\!\sigma}}\mbox{.}
\end{equation}

However, by (\ref{DHLOPP.sec.2.thm.1.eq.2}), $d$ is at least $(\sigma+1)\delta$. Therefore, $P$ and $Q$ are also lattice $(d,k)$\nobreakdash-polytopes. Moreover, replacing $\sigma$ and $\delta$ in the right-hand side of (\ref{DHLOPP.sec.2.thm.1.eq.3}) by their expressions as functions of $d$ and $\beta$ yields
\begin{equation}\label{DHLOPP.sec.2.thm.1.eq.4}
d(P,Q)\leq\frac{\sqrt{\left\lfloor{d^{\beta}}\right\rfloor\left\lfloor\frac{d}{\left\lfloor{d^{\beta}}\right\rfloor+1}\right\rfloor}}{k^{\left\lfloor{d^{\beta}}\right\rfloor}\!\left(\left\lfloor\frac{d}{\left\lfloor{d^{\beta}}\right\rfloor+1}\right\rfloor-1\right)^{\!\left\lfloor{d^{\beta}}\right\rfloor}}\mbox{.}
\end{equation}

Now observe that the right-hand side of (\ref{DHLOPP.sec.2.thm.1.eq.4}) behaves like
$$
\frac{\sqrt{d}}{k^{d^\beta}d^{(1-\beta)d^\beta}}
$$
as $d$ goes to infinity. Since $\alpha$ is less than $\beta$,
$$
\frac{\sqrt{d}}{k^{d^\beta}d^{(1-\beta)d^\beta}}<\frac{1}{k^{d^\alpha}d^{(1-\alpha)d^\alpha}}
$$
when $d$ is large enough. Hence, the right-hand side of (\ref{DHLOPP.sec.2.thm.1.eq.4}) is less than
$$
\frac{1}{k^{d^\alpha}d^{(1-\alpha)d^\alpha}}
$$
for any large enough $d$, as desired.
\end{proof}

It should be noted that, taking $\alpha$ equal to $1/2$ in the statement of Theorem~\ref{DHLOPP.sec.2.thm.1} immediately results in Theorem~\ref{DHLOPP.sec.0.thm.2}.

From now on, we denote $\delta(\sigma+1)$ by $d$. Let us proceed to building the lattice $(d,k)$\nobreakdash-polytopes $P$ and $Q$ that appear in the statement of Theorem \ref{DHLOPP.sec.2.thm.2}.

Denote by $a$ the vector from $\mathbb{Z}^{\sigma+1}$ whose coordinate $i$ is
$$
a_i=\bigl(k(1-\delta)\bigr)^{i-1}\mbox{.}
$$

A vector $\overline{x}$ in $\mathbb{R}^d$ can be built from any vector $x$ in $\mathbb{R}^{\sigma+1}$ by taking
$$
\overline{x}_i=x_{\lfloor{(i-1)/\delta}\rfloor+1}
$$
for every integer $i$. Equivalently,
$$
\overline{x}=(\underbrace{x_1,\ldots,x_1}_{\displaystyle\delta\mbox{ times}},\underbrace{x_2,\ldots,x_2}_{\displaystyle\delta\mbox{ times}},\ldots,\underbrace{x_{\sigma+1},\ldots,x_{\sigma+1}}_{\displaystyle\delta\mbox{ times}})\mbox{.}
$$

Denote by $P$ the convex hull of the lattice points $x$ contained in the hypercube $[0,k]^d$ that satisfy $\overline{a}\mathord{\cdot}x=0$. Likewise, denote by $Q$ the convex hull of the lattice points $x$ in $[0,k]^d$ such that $\overline{a}\mathord{\cdot}x=1$. In order to prove that $P$ and $Q$ satisfy the inequality (\ref{DHLOPP.sec.2.eq.1}), we will exhibit a point in $P$ and a point in $Q$ whose distance is at most the right-hand side of this inequality.

Consider the $(\sigma+1)\mathord{\times}(\sigma+1)$ matrix
$$
M_P=\left[
\begin{array}{cccccc}
0 & A & A & \cdots & A\\
0 & B & C & \cdots & C\\
0& 0 & B & \ddots & \vdots \\
\vdots & \vdots & \ddots & \ddots & C \\
0 & 0 & \cdots & 0 & B \\
\end{array}
\right]
$$
where
$$
\left\{
\begin{array}{l}
A=\displaystyle(\delta-1)k/\delta\mbox{,}\\
B=1/\delta\mbox{, and}\\
C=A+B\mbox{.}\\
\end{array}
\right.
$$

Recall that we identify the points from $\mathbb{R}^{\sigma+1}$ to the vector of their coordinates. In particular, the columns of $M_P$ are points from $\mathbb{R}^{\sigma+1}$.

\begin{prop}\label{DHLOPP.sec.2.eq.2.prop.1}
If $x$ is a column of $M_P$, then $\overline{x}$ belongs to $P$.
\end{prop}
\begin{proof}
Let $x$ be the column $i$ of $M_p$.
Observe that, if $i$ is equal to $1$, then $\overline{a}\mathord{\cdot}\overline{x}=0$ and in particular, $\overline{x}$ belongs to $P$. Now assume that $i$ is at least $2$ and consider an integer $s$ such that $1\leq{s}\leq\delta$. Denote by $u^s$ the lattice point in $[0,k]^d$ whose coordinates are given by
$$
u^s_j=\left\{
\begin{array}{l}
k\mbox{ if }1\leq{j}\leq\delta(i-1)\mbox{ and }\bigl((j-1)\mbox{ mod }{\delta}\bigr)+1\neq{s}\mbox{,}\\
1\mbox{ if }\delta<j\leq\delta{i}\mbox{ and }\bigl((j-1)\mbox{ mod }{\delta}\bigr)+1=s\mbox{, and}\\
0\mbox{ otherwise.}
\end{array}
\right.
$$

Note that $u^s$ is a point in $P$ because $\overline{a}\mathord{\cdot}u^s=0$. As the barycenter of the points $u^s$ when $s$ ranges from $1$ to $\delta$ is precisely $\overline{x}$ this proves the proposition.
\end{proof}

Now consider the $(\sigma+1)\mathord{\times}(\sigma+1)$ matrix $M_Q$ obtained from $M_P$ by adding $1/\delta$ to all the entries of the first row:
$$
M_Q=\left[
\begin{array}{cccccc}
B & C & C & \cdots & C\\
0 & B & C & \cdots & C\\
0& 0 & B & \ddots & \vdots \\
\vdots & \vdots & \ddots & \ddots & C\\
0 & 0 & \cdots & 0 & B \\
\end{array}
\right]\mbox{.}
$$

\begin{prop}\label{DHLOPP.sec.2.eq.2.prop.2}
If $x$ is a column of $M_Q$, then $\overline{x}$ belongs to $Q$.
\end{prop}
\begin{proof}
Let $x$ be the column $i$ of $M_Q$. Consider an integer $s$ such that $1\leq{s}\leq\delta$ and denote by $v^j$ the lattice point in $[0,k]^d$ whose coordinates are
$$
v^s_j=\left\{
\begin{array}{l}
k\mbox{ if }1\leq{j}\leq\delta(i-1)\mbox{ and }\bigl((j-1)\mbox{ mod }{\delta}\bigr)+1\neq{s}\mbox{,}\\
1\mbox{ if }1\leq{j}\leq\delta{i}\mbox{ and }\bigl((j-1)\mbox{ mod }{\delta}\bigr)+1=s\mbox{, and}\\
0\mbox{ otherwise.}
\end{array}
\right.
$$

By construction, $\overline{a}\mathord{\cdot}v^s=1$ and $v^s$ is a point in $Q$. The proposition then follows from the observation that $\overline{x}$ is the barycenter of the points $v^1$ to $v^\delta$.
\end{proof}

For any integer $i$ such that $0\leq{i}\leq\sigma$, we denote the column $i+1$ of the matrix $M_P$ by $p^i$ and the column $i+1$ of $M_Q$ by $q^i$.

Assume that $\delta$ is at least $3$ and consider the points
\begin{equation}\label{DHLOPP.sec.2.eq.2}
p=\Biggl(1-\theta\frac{\bigl(k(\delta-1)\bigr)^{\!\sigma}-1}{\bigl(k(\delta-1)-1\bigr)\bigl(k(\delta-1)\bigr)^{\!\sigma}}\Biggr)p^0+\sum_{i=1}^\sigma\frac{\theta}{\bigl(k(\delta-1)\bigr)^{\!i}}p^i
\end{equation}
and
\begin{equation}\label{DHLOPP.sec.2.eq.3}
q=\Biggl(\frac{k(\delta-1)+1}{k(\delta-1)}-\theta\frac{\bigl(k(\delta-1)\bigr)^{\!\sigma}-1}{\bigl(k(\delta-1)-1\bigr)\bigl(k(\delta-1)\bigr)^{\!\sigma}}\Biggr)q^0+\sum_{i=1}^\sigma\frac{\theta(1+(-1)^i)}{\bigl(k(\delta-1)\bigr)^{\!i}}q^i
\end{equation}
where
$$
\theta=\frac{\bigl(k(1-\delta)-1\bigr)\bigl(k(1-\delta)\bigr)^{\!\sigma-1}}{\bigl(k(1-\delta)\bigr)^{\!\sigma}-1}\mbox{.}
$$

These points are defined as linear combinations of the columns of $M_P$ and $M_Q$. If $\delta$ is at least $4$, they are convex combinations of these columns.

\begin{prop}\label{DHLOPP.sec.2.eq.2.prop.3}
If $\delta$ is at least $4$, then $p$ and $q$ are convex combinations of the columns of $M_P$ and $M_Q$, respectively.
\end{prop}
\begin{proof}
It suffices to show that the coefficients in the right-hand sides of the equations (\ref{DHLOPP.sec.2.eq.2}) and (\ref{DHLOPP.sec.2.eq.3}) are non-negative and sum to $1$. Assume that $\delta$ is at least $4$. In that case, $\theta$ is non-zero and its inverse is
\begin{equation}\label{DHLOPP.sec.2.eq.2.lem.1.eq.1}
\frac{1}{\theta}=\frac{1}{k(\delta-1)+1}\Biggl(k(\delta-1)+\frac{1}{\bigl(k(1-\delta)\bigr)^{\!\sigma-1}}\Biggr)\!\mbox{.}
\end{equation}

Since $\sigma$ is positive,
\begin{equation}\label{DHLOPP.sec.2.eq.2.lem.1.eq.1.5}
\left|\frac{1}{\bigl(k(1-\delta)\bigr)^{\!\sigma-1}}\right|\leq1\mbox{.}
\end{equation}

It follows from (\ref{DHLOPP.sec.2.eq.2.lem.1.eq.1}) and (\ref{DHLOPP.sec.2.eq.2.lem.1.eq.1.5}) that $1/\theta$ and therefore $\theta$ are positive numbers. Hence, all the coefficients in the right-hand sides of (\ref{DHLOPP.sec.2.eq.2}) and (\ref{DHLOPP.sec.2.eq.3}) are non-negative, except possibly for the coefficient of $p^0$ in (\ref{DHLOPP.sec.2.eq.2}) and the coefficient of $q^0$ in (\ref{DHLOPP.sec.2.eq.3}). However, observe that (\ref{DHLOPP.sec.2.eq.2.lem.1.eq.1}) implies
$$
\frac{1}{\theta}\geq\frac{1}{k(\delta-1)+1}\Biggl(k(\delta-1)-\frac{1}{\bigl(k(\delta-1)\bigr)^{\!\sigma-1}}\Biggr)
$$
and, as a consequence,
$$
\theta\leq\frac{\bigl(k(\delta-1)+1\bigr)\bigl(k(\delta-1)\bigr)^{\!\sigma-1}}{\bigl(k(\delta-1)\bigr)^{\!\sigma}-1}\mbox{.}
$$

It follows that the coefficient of $p^0$ in the right-hand side of (\ref{DHLOPP.sec.2.eq.2}) is at least
$$
1-\frac{k(\delta-1)+1}{\bigl(k(\delta-1)-1\bigr)k(\delta-1)}\mbox{.}
$$

This expression is positive when $k(\delta-1)$ is greater than $2$. Hence, the coefficient of $p^0$ in the right-hand side of (\ref{DHLOPP.sec.2.eq.2}) is positive when $k$ is at least $4$. Likewise, the coefficient of $q^0$ in the right-hand side of (\ref{DHLOPP.sec.2.eq.3}) is at least
$$
\frac{1}{k(\delta-1)}+1-\frac{k(\delta-1)+1}{\bigl(k(\delta-1)-1\bigr)k(\delta-1)}
$$
which is positive as well when $k(\delta-1)$ is greater than $2$. Now observe that
$$
\sum_{i=1}^\sigma\frac{1}{\bigl(k(\delta-1)\bigr)^{\!i}}=\frac{\bigl(k(\delta-1)\bigr)^{\!\sigma}-1}{\bigl(k(\delta-1)-1\bigr)\bigl(k(\delta-1)\bigr)^{\!\sigma}}\mbox{.}
$$

Therefore, the coefficients in the right-hand side of (\ref{DHLOPP.sec.2.eq.2}) sum to $1$ and the coefficients in the right-hand side of (\ref{DHLOPP.sec.2.eq.3}) to
$$
\frac{k(\delta-1)+1}{k(\delta-1)}+\sum_{i=1}^\sigma\frac{\theta}{\bigl(k(1-\delta)\bigr)^{\!i}}\mbox{.}
$$

Finally, observe that
$$
\sum_{i=1}^\sigma\frac{\theta}{\bigl(k(1-\delta)\bigr)^{\!i}}=\theta\frac{\bigl(k(1-\delta)\bigr)^{\!\sigma}-1}{\bigl(k(1-\delta)-1\bigr)\bigl(k(1-\delta)\bigr)^{\!\sigma}}=\frac{1}{k(1-\delta)}\mbox{.}
$$

Hence, the coefficients in the right-hand side of (\ref{DHLOPP.sec.2.eq.2}) also sum to $1$.
\end{proof}

We are now ready to bound the distance between $P$ and $Q$. Note that the following theorem immediately implies Theorem \ref{DHLOPP.sec.2.thm.2}.

\begin{thm}\label{DHLOPP.sec.2.thm.4}
If $\delta$ is at least $4$, then
$$
d(P,Q)\leq\frac{\sqrt{\delta\sigma}}{\bigl(\!k(\delta-1)\bigr)^{\!\sigma}}\mbox{.}
$$
\end{thm}

\begin{proof}
According to Propositions \ref{DHLOPP.sec.2.eq.2.prop.1}, \ref{DHLOPP.sec.2.eq.2.prop.2}, and \ref{DHLOPP.sec.2.eq.2.prop.3}, the points $\overline{p}$ and $\overline{q}$ are contained in $P$ and $Q$, respectively. Therefore,
$$
d(P,Q)\leq{d(\overline{p},\overline{q})}\mbox{.}
$$

Now observe that, by construction,
$$
d(\overline{p},\overline{q})=\sqrt{\delta}d(p,q)\mbox{.}
$$

Hence, it suffices to show that
$$
d(p,q)\leq\frac{\sqrt{\sigma}}{\bigl(\!k(\delta-1)\bigr)^{\!\sigma}}\mbox{.}
$$

By (\ref{DHLOPP.sec.2.eq.2}) and (\ref{DHLOPP.sec.2.eq.3}), the first coordinate of $q-p$ is
$$
\begin{array}{l}
\displaystyle q_1-p_1=\frac{k(\delta-1)+1}{\delta{k}(\delta-1)}-\theta\frac{\bigl(k(\delta-1)\bigr)^{\!\sigma}-1}{\delta\bigl(k(\delta-1)-1\bigr)\bigl(k(\delta-1)\bigr)^{\!\sigma}}\hspace{2cm}\\[\bigskipamount]
\displaystyle\hfill+\frac{\theta}{\delta}\sum_{i=1}^\sigma\frac{1}{\bigl(k(\delta-1)\bigr)^{\!i}}+\frac{k(\delta-1)+1}{\delta}\sum_{i=1}^\sigma\frac{\theta}{\bigl(k(1-\delta)\bigr)^{\!i}}\mbox{.}
\end{array}
$$

However, since
$$
\sum_{i=1}^\sigma\frac{1}{\bigl(k(\delta-1)\bigr)^{\!i}}=\frac{\bigl(k(\delta-1)\bigr)^{\!\sigma}-1}{\bigl(k(\delta-1)-1\bigr)\bigl(k(\delta-1)\bigr)^{\!\sigma}}
$$
and
$$
\sum_{i=1}^\sigma\frac{\theta}{\bigl(k(1-\delta)\bigr)^{\!i}}=\frac{1}{k(1-\delta)}
$$
the first coordinate of $q-p$ is equal to $0$. According to (\ref{DHLOPP.sec.2.eq.2}) and (\ref{DHLOPP.sec.2.eq.3}) again, for any integer $j$ satisfying $1\leq{j}\leq\sigma$,
$$
q_{j+1}-p_{j+1}=\frac{1}{\delta}\frac{\theta}{\bigl(k(1-\delta)\bigr)^{j}}+\frac{k(\delta-1)+1}{\delta}\sum_{i=j+1}^\sigma\frac{\theta}{\bigl(k(1-\delta)\bigr)^{\!i}}
$$

However,
$$
\sum_{i=j+1}^\sigma\frac{1}{\bigl(k(1-\delta)\bigr)^{\!i}}=\frac{1-\bigl(k(1-\delta)\bigr)^{\!\sigma-j}}{(k(\delta-1)+1\bigr)\bigl(k(1-\delta)\bigr)^{\!\sigma}}
$$
and as a consequence,
$$
q_{j+1}-p_{j+1}=\frac{\theta}{\delta\bigl(k(1-\delta)\bigr)^{\!\sigma}}=\frac{k(\delta-1)+1}{\delta{k}(1-\delta)\Bigl(1-\bigl(k(1-\delta)\bigr)^{\!\sigma}\Bigr)}
$$

This quantity can be bounded as
$$
|q_j-p_j|\leq\frac{1}{(\delta-1)\Bigl(\bigl(k(\delta-1)\bigr)^{\!\sigma}-1\Bigr)}\leq\frac{1}{\bigl(k(\delta-1)\bigr)^{\!\sigma}}
$$
and therefore,
$$
d(p,q)\leq\frac{\sqrt{\sigma}}{\bigl(k(\delta-1)\bigr)^{\!\sigma}}
$$
as desired.
\end{proof}

\section{Special cases}\label{DHLOPP.sec.3}

From now on, $\varepsilon(d,k)$ denotes the smallest possible distance between two disjoint lattice $(d,k)$-polytopes. In this section, we focus on certain relevant special cases. The upper bounds stated in Section~\ref{DHLOPP.sec.2} imply that $\varepsilon(d,k)$ decreases exponentially fast with $d$ but these bounds only hold when $d$ is \emph{large enough}. We will prove a different bound that holds for all $d$ at least $2$, according to which $\varepsilon(d,1)$ is at most inverse linear as a function of $d$. We shall see in Section~\ref{DHLOPP.sec.4} that this bound on $\varepsilon(d,1)$ is tight when $d$ is equal to $2$ or $3$.

\begin{lem}\label{DHLOPP.sec.3.lem.1}
For any $d$ at least $2$,
$$
\varepsilon(d,1)\leq\frac{1}{\sqrt{d(d-1)}}\mbox{.}
$$
\end{lem}
\begin{proof}
Let $P$ be the diagonal of the hypercube $[0,1]^d$ that is incident to the origin of $\mathbb{R}^d$. Denote by $Q$ the $(d-2)$-dimensional simplex whose vertices are the points $x$ of $\mathbb{R}^d$ whose one of the first $d-1$ coordinates is equal to $1$ and whose all other coordinates are equal to $0$. Note that $P$ and $Q$ are disjoint as the only point of $P$ whose last coordinate is equal to $0$ is the origin of $\mathbb{R}^d$.

The point $p$ of $\mathbb{R}^d$ whose all coordinates are equal to $1/d$ belongs to $P$. The centroid of $Q$ is the point $q$ whose last coordinate is $0$ and whose other coordinates are all equal to $1/(d-1)$. Since
$$
d(p,q)=\frac{1}{\sqrt{d(d-1)}},
$$
this proves the lemma.
\end{proof}

\begin{figure}[t]
\begin{centering}
\includegraphics[scale=1]{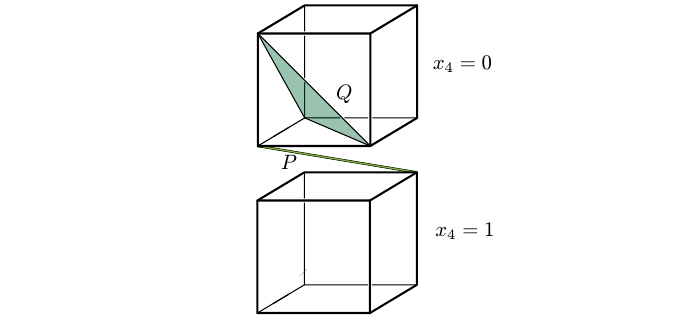}
\caption{The construction of Lemma \ref{DHLOPP.sec.3.lem.1} when $d$ is equal to~$4$. The cube at the top is the facet of the hypercube $[0,1]^4$ made up of the points $x$ such that $x_4=0$ and the cube at the bottom is the opposite facet of  $[0,1]^4$.}\label{DHLOPP.sec.3.fig.1}
\end{centering}
\end{figure}

We complement Lemma \ref{DHLOPP.sec.3.lem.1} by showing that $\varepsilon(d,k)$ is at most inverse linear as a function of $d$ and as a function of $k$ for all $d$ and $k$ at least $2$.

\begin{lem}\label{DHLOPP.sec.3.lem.2}
For any $k$ and $d$ at least $2$,
$$
\varepsilon(d,k)\leq\frac{1}{(d-1)k}\mbox{.}
$$
\end{lem}
\begin{proof}
Let $P$ denote the point of $\mathbb{R}^d$ whose all coordinates are equal to $1$. Denote by $Q$ the $(d-1)$-dimensional simplex whose vertices are the origin of $\mathbb{R}^d$ and the points whose one of the first $d-1$ coordinates is equal to $k-1$ and whose all other coordinates are equal to $k$. Now consider the point $q$ such that
$$
q_i=1-\frac{k}{(d-1)k^2+\bigl((d-1)k-1\bigr)^2}
$$
when $1\leq{i}\leq{d-1}$ and
$$
q_d=1+\frac{(d-1)k-1}{(d-1)k^2+\bigl((d-1)k-1\bigr)^2}\mbox{.}
$$

This point is the convex combination of the vertices of $Q$ where the coefficient of the origin is $1-q_d/k$ and the coefficient of all the other vertices of $Q$ is $q_d/(k(d-1))$. The distance of $P$ and $q$ is
$$
d(P,q)=\frac{1}{\sqrt{(d-1)k^2+\bigl((d-1)k-1\bigr)^2}}\mbox{.}
$$

It suffices to observe that
$$
\frac{1}{\sqrt{(d-1)k^2+\bigl((d-1)k-1\bigr)^2}}\leq\frac{1}{(d-1)k}
$$
when $k\geq2$ in order to complete the proof.
\end{proof}

Let us now turn our attention to the case when the dimensions of $P$ and $Q$ are fixed  independently on the dimension of the ambient space as, for example when $P$ and $Q$ are two line segments that live in a higher dimensional space.

We recall that the dimension of a non-necessarily convex subset of $\mathbb{R}^d$ is defined as the dimension of its affine hull.

\begin{lem}\label{DHLOPP.sec.3.lem.3}
For any two disjoint lattice $(d,k)$-polytopes $P$ and $Q$,
$$
d(P,Q)\geq\varepsilon\bigl(\mathrm{dim}(P\cup{Q}),k\bigr)\mbox{.}
$$
\end{lem}
\begin{proof}
The proof is by induction on $d-\mathrm{dim}(P\cup{Q})$. If this quantity is equal to $0$, then the result is immediate. Let us assume that $d$ is greater than the dimension of $P\cup{Q}$. In that case, there exists an hyperplane $H$ of $\mathbb{R}^d$ that contains $P$ and $Q$. Identify $\mathbb{R}^{d-1}$ with the subspace of $\mathbb{R}^d$ spanned by the first $d-1$ coordinates. We can assume that the vectors orthogonal to $H$ do not belong to $\mathbb{R}^{d-1}$ by, if needed using an adequate permutation of the coordinates of $\mathbb{R}^d$. Now consider the orthogonal projection $\pi:\mathbb{R}^d\rightarrow\mathbb{R}^{d-1}$. Since the vectors orthogonal to $H$ do not belong to $\mathbb{R}^{d-1}$, the restriction of $\pi$ to $H$ is a bijection between $H$ and $\mathbb{R}^{d-1}$. Moreover, $\pi(\mathbb{Z}^d\cap{H})$ is a subset of $\mathbb{Z}^{d-1}$. Hence, $\pi(P)$ and $\pi(Q)$ are two disjoint lattice $(d-1,k)$-polytopes and the dimension of $\pi(P)\cup\pi(Q)$ coincides with the dimension of $P\cup{Q}$. In particular,
$$
d-1-\mathrm{dim}\bigl(\pi(P)\cup\pi(Q)\bigr)=d-\mathrm{dim}\bigl(P\cup{Q}\bigr)-1\mbox{.}
$$

By induction,
\begin{equation}\label{DHLOPP.sec.3.lem.3.eq.1}
d\bigl(\pi(P),\pi(Q)\bigr)\geq\varepsilon\Bigl(\mathrm{dim}\bigl(\pi(P)\cup\pi(Q)\bigr),k\Bigr)=\varepsilon\bigl(\mathrm{dim}(P\cup{Q}),k\bigr)\mbox{.}
\end{equation}

Finally observe that the distance between two points in $H$ is always at least the distance between their images by $\pi$. Therefore,
$$
d(P,Q)\geq{d\bigl(\pi(P),\pi(Q)\bigr)}
$$
and combining this with (\ref{DHLOPP.sec.3.lem.3.eq.1}) proves the lemma.
\end{proof}

We will see in Section \ref{DHLOPP.sec.4} that $\varepsilon(3,1)$ is equal to $1/\sqrt{6}$ (see for instance Table~\ref{DHLOPP.sec.4.table.1}) and that this distance is achieved between a diagonal of the cube $[0,1]^3$ and a diagonal of one of its square faces. An immediate consequence of Lemma \ref{DHLOPP.sec.3.lem.3} is that this holds independently on the dimension of the ambient space.

\begin{thm}\label{DHLOPP.sec.3.thm.2}
The smallest possible distance between two disjoint line segments whose vertices belong to $\{0,1\}^d$ is $1/\sqrt{6}$.
\end{thm}

\section{Computational aspects}\label{DHLOPP.sec.4}

In this section, we are interested in computing the explicit value of $\varepsilon(d,k)$, the smallest between two disjoint lattice $(d,k)$-polytopes. A strategy is to enumerate all possible pairs of disjoint lattice $(d,k)$\nobreakdash-polytopes. Let us give some properties of that allow to reduce the search space. 

By its definition, $\varepsilon(d,k)$ is a non-increasing function of $d$ for all fixed $k$. We can prove the following stronger statement.

\begin{thm}\label{DHLOPP.sec.3.thm.1}
$\varepsilon(d,k)$ is a decreasing function of $d$ for all fixed $k$.
\end{thm}
\begin{proof}
Let us identify $\mathbb{R}^{d-1}$ with the subspace of $\mathbb{R}^d$ spanned by the first $d-1$ coordinates. Consider two lattice $(d-1,k)$-polytopes $P$ and $Q$ such that $d(P,Q)$ is equal to $\varepsilon(d-1,k)$. Now consider the map $\phi:\mathbb{R}^{d-1}\rightarrow\mathbb{R}^d$ such that $\phi(x)$ is the point of $\mathbb{R}^d$ obtained from $x$ by changing its last coordinate to $1$.

Now consider the lattice $(d,k)$-polytope
$$
Q'=\mathrm{conv}\bigl(\phi(P)\cup{Q}\bigr)\mbox{.}
$$

Denote by $p$ and $q$ a point in $P$ and a point in $Q$ whose distance is equal to $\varepsilon(d-1,k)$. By construction, both $q$ and $\phi(p)$ belong to $Q'$. Now consider a number $\lambda$ in the interval $[0,1]$ and denote by $\delta$ the squared distance between the points $p$ and $\lambda\phi(p)+(1-\lambda)q$. It should be noted that $\delta$ coincides with $d(p,q)^2$ when $\lambda$ is equal to $0$. Observe that
$$
\delta=(1-\lambda)^2d(p,q)^2+\lambda^2\mbox{.}
$$

Differentiating this equality with respect to $\lambda$ yields
$$
\frac{\partial\delta}{\partial\lambda}=2\lambda\Bigl(1+d(p,q)^2\Bigr)-2d(p,q)^2\mbox{.}
$$

Note that this derivative is negative for all $\lambda$ close enough to $0$. In particular, one can find a value of $\lambda$ such that $\delta$ is less than $d(p,q)^2$. As $\delta$ is the squared distance between $p$ and a point in $Q'$, this show that
$$
d(P,Q')<d(p,q)
$$

Since the right-hand side of this inequality is equal to $\varepsilon(d-1,k)$ and its left hand side is at least $\varepsilon(d,k)$, this proves the lemma.
\end{proof}

By the following theorem, in order to compute $\varepsilon(d,k)$ by enumerating all possible pairs of lattice $(d,k)$\nobreakdash-polytopes, one only needs to consider pairs of disjoint simplices whose dimensions sum to $d-1$.

\begin{thm}\label{DHLOPP.sec.4.lem.1}
There exist two lattice $(d,k)$-polytopes $P$ and $Q$ such that
\begin{enumerate}
\item[(i)] $d(P,Q)$ is equal to $\varepsilon(d,k)$,
\item[(ii)] both $P$ and $Q$ are simplices,
\item[(iii)] $\mathrm{dim}(P)+\mathrm{dim}(Q)$ is equal to $d-1$, and
\item[(iv)] the affine hulls of $P$ and $Q$ are disjoint.
\end{enumerate}
\end{thm}
\begin{proof}
Consider two disjoint lattice $(d,k)$-polytopes $P$ and $Q$ such that $d(P,Q)$ is equal to $\varepsilon(d,k)$. Among all possible such pairs of polytopes, we choose $P$ and $Q$ in such a way that their number of vertices sum to the smallest possible number. We shall prove that a consequence of this choice is that $P$ and $Q$ satisfy assertions (ii) and (iii) in the statement of the lemma.

Consider a point $p$ in $P$ and a point $q$ in $Q$ such that $d(p,q)$ is equal to $d(P,Q)$. By Carath{\'e}odory's theorem, $p$ is a convex combination of a set $S_P$ of at most $\mathrm{dim}(P)+1$ affinely independent vertices of $P$. Moreover, we can choose $S_P$ in such a way that all the points it contains have a positive coefficient in that convex combination. Equivalently, $p$ lies in the relative interior of $\mathrm{conv}(S_P)$. In that case, $\varepsilon(d,k)$ is achieved as the distance between $\mathrm{conv}(S_P)$ and $Q$. It then follows from the above choice for $P$ and $Q$ that $S_P$ must is precisely the vertex set of $P$. As a consequence, $P$ is a simplex that contains $p$ in its relative interior. By the same argument, $Q$ is a simplex a well and $q$ lies in its relative interior which proves assertion (ii).

Let us now turn our attention to assertion (iii). First observe that if $\mathrm{dim}(P)+\mathrm{dim}(Q)$ is less than $d-1$, then $\mathrm{dim}(P\cup{Q})$ is at most $d-1$ and by Lemma \ref{DHLOPP.sec.3.lem.3}, $d(P,Q)\geq\varepsilon(d-1,k)$, which would contradict Theorem \ref{DHLOPP.sec.3.thm.1} because $d(P,Q)$ is equal to $\varepsilon(d,k)$. This shows that $\mathrm{dim}(P)+\mathrm{dim}(Q)$ is at least $d-1$. Let us now show that the opposite inequality holds.

By convexity, one can associate a positive number $\alpha_u$ with each point in $S_P\cup{S_Q}$ in such a way that these numbers collectively satisfy
$$
\left\{
\begin{array}{l}
\displaystyle\sum_{u\in{S_P}}\alpha_uu=p\mbox{,}\\[\bigskipamount]
\displaystyle\sum_{u\in{S_P}}\alpha_u=1\mbox{,}\\
\end{array}\right.
$$
and the same equalities hold when $S_P$ is replaced by $S_Q$ and $p$ by $q$. Now consider a vertex $v_P$ of $P$, a vertex $v_Q$ of $Q$. As $P$ and $Q$ are simplices, the sets
$$
S'_P=\Bigl\{u-v_P:u\in{S_P\mathord{\setminus\{v_P\}}}\Bigr\}
$$
and
$$
S'_Q=\Bigl\{u-v_Q:u\in{S_Q\mathord{\setminus\{v_Q\}}}\Bigr\}\mbox{.}
$$
are linearly independent. Further observe that all the vectors they contain are orthogonal to $p-q$. Hence, these vectors collectively span a linear subspace $M$ of $\mathbb{R}^d$ of dimension at most $d-1$. Assume for contradiction that the dimensions of $P$ and $Q$ sum to at least $d$. In that case, the dimensions of the subspaces of $M$ spanned by $S'_P$ and by $S'_Q$ also sum to at least $d$ and the intersection of these subspaces has dimension at least one. Let $x$ be a non-zero point in that intersection. This point can be expressed as a linear combination of $S'_P$: one can associate each point $u$ in $S_P\mathord{\setminus}\{v_P\}$ with a number $\beta_u$ such that
$$
\sum_{u\in{S_P}\mathord{\setminus}\{v_P\}}\beta_u(u-v_P)=x\mbox{.}
$$

As $x$ is non-zero, the coefficients in the left-hand side of this equality cannot all be equal to zero. For any $u$ in $S_P$, denote $\gamma_u=\beta_u$ when $u\neq{v_P}$ and
$$
\gamma_u=-\sum_{u\in{S_P}\mathord{\setminus}\{v_P\}}\beta_u
$$
when $u=v_P$. With these notations,
\begin{equation}\label{DHLOPP.sec.4.lem.1.eq.1}
\sum_{u\in{S_P}}\gamma_u=0
\end{equation}
and
\begin{equation}\label{DHLOPP.sec.4.lem.1.eq.2}
\sum_{u\in{S_P}}\gamma_uu=x\mbox{.}
\end{equation}

Likewise, one can associate each point $u$ in $S_Q$ with a number $\gamma_u$ such that (\ref{DHLOPP.sec.4.lem.1.eq.1}) and (\ref{DHLOPP.sec.4.lem.1.eq.1}) still hold when replacing $S_P$ by $S_Q$.

Now consider the number
$$
\lambda=\mathrm{min}\left\{\frac{\alpha_u}{\gamma_u}:u\in{S_P}\cup{S_Q},\,\gamma_u>0\right\}
$$

It follows from this choice for $\lambda$ that the point $p-\lambda{x}$ is still contained in $P$ because the coefficients of is decomposition into an affine combination of $S_P$ all remain non-negative. Likewise $q-\lambda{x}$ still belongs to $Q$. Further observe that the distance between $p-\lambda{x}$ and $q-\lambda{x}$ is still equal to $\varepsilon(d,k)$. However, also by our choice for $\lambda$, at least one of the coefficients in the expression of $p-\lambda{x}$ as a convex combination of $S_P$ or in the expression of $q-\lambda{x}$ as a convex combination of $S_Q$ must vanish. In other words, $\varepsilon(d,k)$ is achieved by a pair of disjoint lattice simplices whose combined number of vertices is less than that of $P$ and $Q$. This contradicts the assumption that $P$ and $Q$ have the smallest combined number of vertices among the pairs of disjoint lattice $(d,k)$-polytopes whose distance is equal to $\varepsilon(d,k)$, which proves assertion (iii).

Finally, in order to prove (iv), observe that the affine hulls of $P$ and $Q$ are contained in two hyperplanes of $\mathbb{R}^d$ orthogonal to $p-q$. These two hyperplanes are disjoint because they are parallel and one of them contains $p$ while the other contains $q$. As a consequence, (iv) holds, as desired.
\end{proof}

Using Theorem \ref{DHLOPP.sec.4.lem.1}, one can compute $\varepsilon(d,k)$. This is done in practice by generating all the subsets of at most $d$ points from $\{0,1,\ldots,k\}^d$, by computing the dimension of their affine hull, and by discarding the subsets such that this dimension is less by at least $2$ than the number of points they contain. Then, all the possible pairs of these subsets are considered whose affine hull dimensions sum to $d-1$ and the distance of their convex hulls is computed. The requirement that the dimensions of the considered pairs sum to $d-1$ can be made without loss of generality thanks to Theorem \ref{DHLOPP.sec.4.lem.1}. This procedure can be further sped up by doing the computation up to the symmetries of $[0,k]^d$. This allowed to determine the values of $\varepsilon(d,k)$ reported in Table~\ref{DHLOPP.sec.4.table.1}.

Let us provide two lattice $(d,k)$-polytopes that achieve each of the values of $\varepsilon(d,k)$ reported in that table. The smallest possible distance between disjoint lattice $(2,1)$-polytopes is achieved by the origin of $\mathbb{R}^2$ and the diagonal of $[0,1]^2$ that doesn't contain the origin. For all the other values of $k$ considered in Table~\ref{DHLOPP.sec.4.table.1} in the two dimensional case, $\varepsilon(2,k)$ is achieved by the point $(1,1)$ and the line segment with vertices $(0,0)$ and $(k,k-1)$. These configurations are illustrated at the top of Figure~\ref{DHLOPP.sec.4.fig.1} when $k$ is equal to $1$, $2$, and $3$.
\begin{table}[t]
\begin{tabular}{>{\centering}p{0.7cm}cccccc}
\multirow{2}{*}{$d$}&  \multicolumn{6}{c}{$k$}\\
\cline{2-7}
 & $1$ & $2$ & $3$ & $4$ & $5$ & $6$\\
\hline & \\[-1.1\bigskipamount]

$2$ & $\sqrt{2}$ & $\sqrt{5}$ & $\sqrt{13}$ & $5$ & $\sqrt{41}$ & $\sqrt{61}$\\
$3$ & $\sqrt{6}$ & $5\sqrt{2}$ & $\sqrt{299}$\\
$4$ & $3\sqrt{2}$ &\\
$5$ & $\sqrt{58}$ & \\[\medskipamount]
\end{tabular}
\caption{A few values of $1/\varepsilon(d,k)$.}\label{DHLOPP.sec.4.table.1}
\end{table}
\begin{figure}[b]
\begin{centering}
\includegraphics[scale=1]{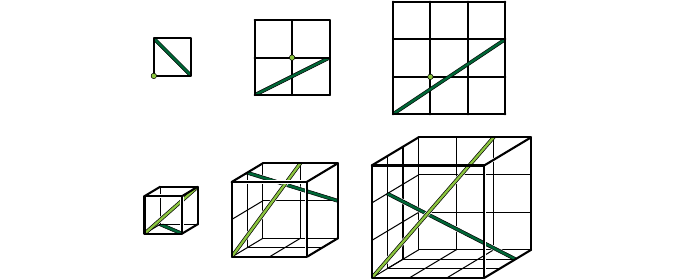}
\caption{Two polytopes $P$ and $Q$ whose distance is equal to $\varepsilon(2,k)$ (top) and $\varepsilon(3,k)$ (bottom) when $1\leq{k}\leq3$.}\label{DHLOPP.sec.4.fig.1}
\end{centering}
\end{figure}

Pairs of line segments whose distance are $\varepsilon(3,1)$, $\varepsilon(3,2)$, and $\varepsilon(3,3)$ are shown at the bottom of Figure~\ref{DHLOPP.sec.4.fig.1}. As already mentioned, $\varepsilon(3,1)$ is achieved by a diagonal of the cube $[0,1]^3$ and a diagonal of a square face. In addition, the line segment with vertices $(0,0,0)$ and $(1,2,2)$ is at distance $\varepsilon(3,2)$ of the segment with vertices $(0,1,2)$ and $(2,2,1)$. Similarly, the line segment with vertices $(0,0,0)$ and $(2,3,3)$ is at distance $\varepsilon(3,3)$ from the segment with vertices $(0,1,2)$ and $(3,2,0)$. In four dimensions, $\varepsilon(4,1)$ is achieved between the diagonal of the hypercube $[0,1]^4$ incident to the origin and the triangle with vertices $(0,0,0,1)$, $(0,1,1,0)$, and $(1,0,1,0)$. In five dimensions, $\varepsilon(5,1)$ is achieved between the diagonal of the hypercube $[0,1]^5$ incident to the origin and the tetrahedron with vertices $(0,0,0,1,1)$, $(0,0,1,0,1)$, $(0,1,1,1,0)$, and $(1,1,0,0,0)$. In the configurations we have just described, $\varepsilon(d,1)$ is always achieved by a diagonal of the hypercube and a $(d-2)$-dimensional simplex. It should be noted that these simplices are not standard simplices and these configurations are therefore different from the ones we used in the proof of Lemma \ref{DHLOPP.sec.3.lem.1}.
 
\section{Estimates in terms of encoding length}\label{DHLOPP.sec.1.5}

We finally turn our attention to bounding $d(P,Q)$ in the case when $P$ and $Q$ are rational polytopes. In practice, the parameter that quantifies the size of a rational polytope is its binary \emph{encoding input data length}. This parameter is the number $L$ of bits required to represent a rational polytope $P$ using either a system of linear inequalities whose set of solutions is $P$ or a set of points whose convex hull is $P$. Note however that $L$ depends on the choice of a representation for $P$ and may grow large in case the representation is redundant. Therefore, we will follow the terminology from \cite{Schrijver1998} and use the vertex and facet complexities of $P$ for our analysis. Let us introduce these two quantities. If $\alpha$ and $\beta$ are two relatively prime integers such that $\beta$ is positive, the \emph{size} of $\alpha/\beta$ is
$$
\mathrm{size}\biggl(\frac{\alpha}{\beta}\biggr)=1+\bigl\lceil{\log_2(|\alpha|+1)}\bigr\rceil+\bigl\lceil{\log_2(\beta+1)}\bigr\rceil\mbox{.}
$$

In turn, the size of a vector $a$ from $\mathbb{R}^d$ with rational coordinates is
$$
\mathrm{size}(a)=d+\sum_{i=1}^d\mathrm{size}(a_i)\mbox{.}
$$

In other words, the size of a vector with rational coordinates is the number of its coordinates plus the sum of the sizes of these coordinates.

If $P$ is a rational polytope, then its vertices have rational coordinates and the \emph{vertex complexity of $P$} is the smallest number $\nu(P)$ such that $\nu(P)$ is at least $d$ and the size of any vertex of $P$ is at most $\nu(P)$. Still under the assumption that $P$ is a rational polytope, the \emph{facet complexity} of $P$ is the smallest number $\varphi(P)$ such that $\varphi(P)$ is at least $d$ and there exists a family of vectors $a^1$ to $a^n$ from $\mathbb{Q}^d$ and a family of rational numbers $b_1$ to $b_n$ such that
$$
P=\bigl\{x\in\mathbb{R}^d:\forall\,i\in\{1,\ldots,d\},\,a^i\mathord{\cdot}x\leq{b_i}\bigr\}
$$
and for all $i$ satisfying $1\leq{i}\leq{n}$,
$$
\mathrm{size}(a^i)+\mathrm{size}(b_i)\leq{\varphi(P)}\mbox{.}
$$

The following is proven in \cite{Schrijver1998} (see Theorem 10.2 therein).

\begin{thm}\label{DHLOPP.sec.1.thm.2}
If $P$ is rational, then $\nu(P)\leq4d^2\varphi(P)$ and $\varphi(P)\leq4d^2\nu(P)$.
\end{thm}

The size of a matrix can be defined in the same spirit as the size of a vector: if $M$ is a matrix with rational coefficients, then $\mathrm{size}(M)$ is the number of coefficients in $M$ plus the sum of the sizes of these coefficients.

The following statement is Theorem 3.2 from \cite{Schrijver1998}.

\begin{thm}\label{DHLOPP.sec.1.thm.3}
If $M$ is a square matrix with rational coefficients, then
$$
\mathrm{size}\bigl(\mathrm{det}(M)\bigr)\leq2\,\mathrm{size}(M)\mbox{.}
$$
\end{thm}

We now state two propositions that allow to bound the sizes of rational numbers or vectors. The first one is given as Exercise 1.3.5 in~\cite{GrotschelLovaszSchrijver1993}).
\begin{prop}\label{DHLOPP.sec.1.prop.1}
If $a$ and $b$ are two vectors from $\mathbb{Q}^d$, then
$$
\mathrm{size}\Biggl(\sum_{i=1}^da_i\Biggr)\leq2\sum_{i=1}^d\mathrm{size}(a_i)
$$
and
$$
\mathrm{size}(a\mathord{\cdot}b)\leq2\,\mathrm{size}(a)+2\,\mathrm{size}(b)\mbox{.}
$$
\end{prop}

The second proposition, whose proof is straightforward provides the smallest possible positive rational number with a given size.

\begin{prop}\label{DHLOPP.sec.1.prop.3}
If $x$ is a positive rational number, then
$$
\frac{4}{2^{\mathrm{size}(x)}}\leq{x}\leq\frac{2^{\mathrm{size}(x)}}{4}\mbox{.}
$$
\end{prop}

We are ready to give lower bounds on the distance of two disjoint rational polytopes $P$ and $Q$ in terms of their binary encoding input data length. In the statement of the following theorem and its proof, we denote
$$
\nu(P,Q)=\mathrm{max}\bigl\{\nu(P),\nu(Q)\bigr\}
$$
and
$$
\varphi(P,Q)=\mathrm{max}\bigl\{\varphi(P),\varphi(Q)\bigr\}\mbox{.}
$$

\begin{thm}\label{DHLOPP.sec.1.thm.4}
If $P$ and $Q$ are disjoint rational polytopes, then
\begin{equation}\label{DHLOPP.sec.1.thm.4.eq.-2}
d(P,Q)\geq\frac{8}{2^{4\nu(P,Q)(2d)^4}}
\end{equation}
and
\begin{equation}\label{DHLOPP.sec.1.thm.4.eq.-1}
d(P,Q)\geq\frac{8}{2^{4\varphi(P,Q)(2d)^6}}\mbox{.}
\end{equation}
\end{thm}
\begin{proof}
In this proof, we consider the vectors $w^0$ to $w^r$ as well as the matrices $M$ and $M_1$ to $M_r$ that have been associated to $P$ and $Q$ at the beginning of Section \ref{DHLOPP.sec.1}. Recall that the vectors $w^0$ to $w^r$ are obtained by subtracting from one another two vertices of $P$, two vertices of $Q$, or a vertex of $P$ and a vertex of $Q$. As a consequence, it follows from the first inequality in the statement of Proposition \ref{DHLOPP.sec.1.prop.1}, that for every integer $i$ satisfying $0\leq{i}\leq{r}$,
$$
\mathrm{size}\bigl(w^i\bigr)\leq2\nu(P,Q)\mbox{.}
$$

In turn, for any two integers $i$ and $j$ satisfying $0\leq{i}\leq{j}\leq{r}$, it follows from Proposition \ref{DHLOPP.sec.1.prop.1} that the size of $w^i\mathord{\cdot}w^j$ can be bounded as
$$
\mathrm{size}\bigl(w^i\mathord{\cdot}w^j\bigr)\leq4\nu(P,Q)
$$
and by Theorem \ref{DHLOPP.sec.1.thm.3},
$$
\mathrm{size}\bigl(\mathrm{det}(M)\bigr)\leq8r^2\nu(P,Q)\mbox{.}
$$

In addition, the same inequality holds when replacing $M$ by any of the matrices $M_1$ to $M_r$. Now consider the vector
$$
a=\mathrm{det}(M)w^0-\sum_{i=1}^r\mathrm{det}(M_i)w^i
$$
and observe that $a_i$ is the scalar product between the vector from $\mathbb{R}^{r+1}$ whose coordinates are $\mathrm{det}(M)$ and $\mathrm{det}(M_1)$ to $\mathrm{det}(M_r)$ with the one whose coordinates are $w_i^0$ and $-w_i^1$ to $-w_i^r$. Therefore, by Proposition \ref{DHLOPP.sec.1.prop.1},
$$
\mathrm{size}(a)\leq2d(8r^2+1)(r+1)\nu(P,Q)
$$
and
$$
\mathrm{size}\bigl(w^0\mathord{\cdot}a\bigr)\leq4\bigl(d(8r^2+1)(r+1)+1\bigr)\nu(P,Q)\mbox{.}
$$

However, recall that $r$ is at most $d-1$. Hence,
\begin{equation}\label{DHLOPP.sec.1.thm.4.eq.1}
\mathrm{size}(a)\leq16d^4\nu(P,Q)
\end{equation}
and
$$
\mathrm{size}\bigl(w^0\mathord{\cdot}a\bigr)\leq32d^4\nu(P,Q)\mbox{.}
$$

In turn, according to Lemma \ref{DPP.sec.0.lem.1} and Proposition \ref{DHLOPP.sec.1.prop.3},
\begin{equation}\label{DHLOPP.sec.1.thm.4.eq.2}
d(P,Q)\geq\frac{4}{2^{32d^4\nu(P,Q)}\|a\|}\mbox{.}
\end{equation}

It also follows from (\ref{DHLOPP.sec.1.thm.4.eq.1}) and Proposition \ref{DHLOPP.sec.1.prop.1} that
$$
\mathrm{size}\bigl(\|a\|^2\bigr)\leq64d^4\nu(P,Q)
$$
and therefore, by Proposition \ref{DHLOPP.sec.1.prop.3},
\begin{equation}\label{DHLOPP.sec.1.thm.4.eq.3}
\|a\|^2\leq\frac{2^{64d^4\nu(P,Q)}}{4}\mbox{.}
\end{equation}

The desired lower bound on $d(P,Q)$ in terms of $\nu(P,Q)$ is obtained by combining (\ref{DHLOPP.sec.1.thm.4.eq.2}) with (\ref{DHLOPP.sec.1.thm.4.eq.3}). Finally, recall that Theorem \ref{DHLOPP.sec.1.thm.2} allows to upper bound $\nu(P,Q)$ by a function of $\varphi(P,Q)$. Using this bound on $\nu(P,Q)$ in the denominator of the right-hand side of (\ref{DHLOPP.sec.1.thm.4.eq.-2}) proves (\ref{DHLOPP.sec.1.thm.4.eq.-1}).
\end{proof}

We can also give upper bounds on the smallest possible distance of two disjoint rational polytopes in terms of their binary encoding input data length. Such bounds can be easily derived from Theorem \ref{DHLOPP.sec.2.thm.1}. It should be noted that these bounds no longer depend on $k$ or $d$. In particular, the following theorem is obtained using two $0/1$\nobreakdash-polytopes whose dimension gets arbitrarily large and even though we do not use the dependence on $k$, their distance decreases exponentially fast with their binary encoding input data length.

\begin{thm}\label{DHLOPP.sec.2.thm.3}
For any number $\alpha$ in $]0,1[$ there exist two disjoint rational polytopes $P$ and $Q$ with arbitrarily large $\nu(P,Q)$ and $\varphi(P,Q)$ such that
$$
d(P,Q)\leq\frac{1}{\displaystyle\biggl(\frac{\nu(P,Q)}{4}\biggr)^{\!(1-\alpha)\bigl(\frac{\nu(P,Q)}{4}\bigr)^\alpha}}
$$
and
$$
d(P,Q)\leq\frac{1}{\displaystyle\biggl(\frac{\varphi(P,Q)}{16}\biggr)^{\!\frac{1-\alpha}{3}\bigl(\frac{\varphi(P,Q)}{16}\bigr)^{\!\frac{\alpha}{3}}}}\mbox{.}
$$
\end{thm}
\begin{proof}
The theorem is proven by rewriting Theorem \ref{DHLOPP.sec.2.thm.1} in terms of the binary encoding input data length of $P$ and $Q$. Indeed, observe that, setting $k$ to $1$, $\nu(P)$ and $\nu(Q)$ are both bounded by $4d$ as the coordinates of the vertices of $P$ and $Q$ are equal to $0$ or to $1$ and the sizes of these two numbers are $2$ and~$3$. As a consequence, $\nu(P,Q)$ is at most $4d$ as well and $\varphi(P,Q)$ at most $16d^3$ according to Theorem~\ref{DHLOPP.sec.1.thm.2} which can be rewritten into
$$
d\geq\frac{\nu(P,Q)}{4}
$$
and
$$
d\geq\biggl(\frac{\varphi(P,Q)}{16}\biggr)^{\!\frac{1}{3}}\mbox{.}
$$

Bounding $d$ using these two inequalities in the denominator of the upper bound on $d(P,Q)$ from Theorem \ref{DHLOPP.sec.2.thm.1} proves the theorem.
\end{proof}

\medskip
\noindent\textbf{Acknowledgements.} This article is based upon work partially supported by the National Science Foundation under Grant No. DMS-1929284 while the authors were in residence at the Institute for Computational and Experimental Research in Mathematics in Providence, Rhode Island, during the Discrete Optimization: Mathematics,  Algorithms, and Computation semester program. The authors are grateful to Amitabh Basu, Santanu Dey, Yuri Faenza, Fr{\'e}d{\'e}ric Meunier, and Dmitrii Pasechnik for useful discussions and for pointing out relevant references. The first author is partially supported by the Natural Sciences and Engineering Research Council of Canada Discovery Grant program number RGPIN-2020-06846, the second author by a grant from the Israel Science Foundation and by the Dresner Chair at the Technion, and the third by the DFG Cluster of Excellence MATH+ (EXC-2046/1, project id 390685689) funded by the Deutsche Forschungsgemeinschaft (DFG).

\bibliography{KissingPolytopes}
\bibliographystyle{ijmart}

\end{document}